\documentclass[a4paper,12pt]{article}
\usepackage{a4wide}
\usepackage{amsmath}
\usepackage{amssymb}
\usepackage{amsthm}
\usepackage{latexsym}
\usepackage{graphicx}
\usepackage[english]{babel}
\usepackage{makeidx}

\newtheorem{obs} [subsection]{Remark}

\newtheorem{prop}[subsection]{Proposition}
\newtheorem{conj}[subsection]{Conjecture}
\newtheorem{teor}[subsection]{Theorem}
\newtheorem*{teor*}{Theorem}
\newtheorem{lema}[subsection]{Lemma}
\newtheorem{cor} [subsection]{Corollary}

\newcommand{\pa}{p_{\mathbf a}}

\newcommand{\za}{\zeta_{\mathbf a}}

\newcommand{\Res}{Res}

\def\Ree{\operatorname{Re}}

\begin{document}
\selectlanguage{english}
\frenchspacing
\numberwithin{equation}{section}

\large
\begin{center}
\textbf{Determinants with Bernoulli polynomials and the restricted partition function}

Mircea Cimpoea\c s
\end{center}
\normalsize

\begin{abstract}
Let $r\geq 1$ be an integer, $\mathbf a=(a_1,\ldots,a_r)$ a vector of positive integers and let $D\geq 1$ be a common multiple of $a_1,\ldots,a_r$.
We study two natural determinants of order $rD$ with Bernoulli polynomials and we present connections with the restricted partition function $p_{\mathbf a}(n):=$
the number of integer solutions $(x_1,\dots,x_r)$ to $\sum_{j=1}^r a_jx_j=n$ with $x_1\geq 0, \ldots, x_r\geq 0$
 

\end{abstract}

\section{Introduction}

Let $\mathbf a:=(a_1,a_2,\ldots,a_r)$ be a sequence of positive integers, $r\geq 1$. 
The \emph{restricted partition function} associated to $\mathbf a$ is $\pa:\mathbb N \rightarrow \mathbb N$, 
$\pa(n):=$ the number of integer solutions $(x_1,\ldots,x_r)$ of $\sum_{i=1}^r a_ix_i=n$ with $x_i\geq 0$.
Let $D$ be a common multiple of $a_1,\ldots,a_r$.
The restricted partition function $\pa(n)$ was studied extensively in the literature, starting with the works of Sylvester \cite{sylvester} and Bell \cite{bell}.
Popoviciu \cite{popoviciu} gave a precise formula for $r=2$. Recently, Bayad and Beck \cite[Theorem 3.1]{babeck} proved an
explicit expression of $\pa(n)$ in terms of Bernoulli-Barnes polynomials and the Fourier Dedekind sums, in the case that
$a_1,\ldots,a_r$ are are pairwise coprime. 

Let $D$ be a common multiple of $a_1,\ldots,a_r$.
In \cite{lucrare}, we reduced the computation of $\pa(n)$ to solving the linear
congruence $a_1j_1+\cdots+a_rj_r \equiv n (\bmod\;D)$ in the range $0\leq j_1\leq \frac{D}{a_1}-1,\ldots,0\leq j_r\leq \frac{D}{a_r}-1$.
In \cite{lucrare2}, we proved that if a determinant $\Delta_{r,D}$, see \eqref{pista}, which
depends only on $r$ and $D$, with entries consisting in values of Bernoulli polynomials
is nonzero, then $\pa(n)$ can be computed in terms of values of Bernoulli polynomials and Bernoulli Barnes numbers.
In the second section, we outline several construction and results from \cite{lucrare2}.
In the third section, we study the polynomial \small{
$$
 F_{r,D}(x_1,\ldots,x_D):= \begin{vmatrix}
\frac{B_1(x_1)}{1} & \cdots & \frac{B_1(x_D)}{1}  
& \cdots & \frac{B_r(x_1)}{r} & \cdots & \frac{B_r(x_D)}{r} \\
\frac{B_2(x_1)}{2} & \cdots & \frac{B_2(x_D)}{2}  
& \cdots & \frac{B_{r+1}(x_1)}{r+1} & \cdots & \frac{B_{r+1}(x_D)}{r+1} \\
\vdots & \vdots & \vdots & \vdots & \vdots & \vdots & \vdots  
\\
\frac{B_{rD}(x_1)}{rD} & \cdots & \frac{B_{rD}(x_D)}{rD} & 
 \cdots & \frac{B_{rD+r-1}(x_1)}{rD+r-1} & \cdots & \frac{B_{rD+r-1}(x_D)}{rD+r-1} 
\end{vmatrix} \in \mathbb Q[x_1,\ldots,x_D],$$}
which is related to $\Delta_{r,D}$ by the identity
$$\Delta_{r,D}=(-1)^{\frac{rD(rD +r)}{2}}D^{\frac{rD(rD +r-2)}{2}}\cdot F_{r,D}(\frac{D-1}{D},\ldots,\frac{1}{D},0).$$
In Theorem $3.4$ we prove that
$$F_{1,D}(x_1,\ldots,x_D)=\frac{1}{D!}  \prod_{1\leq i<j\leq D}(x_j-x_i) \sum_{t=0}^D (-1)^t \frac{E_{D,D-t}(x_1,\ldots,x_D)}{t+1},$$
where $E_{D,0}(x_1,\ldots,x_D)=1$, $E_{D,1}(x_1,\ldots,x_D)=x_1+\cdots+x_D$ etc. are the \emph{elementary symmetric polynomials}. In Proposition $3.6$, we prove that
$$F_{r,D}(x_1,\ldots,x_D) = \left(\prod_{1\leq i<j\leq D}(x_j-x_i)^r\right) G_{r,D}(x_1,\ldots,x_D),$$ 
where $G_{r,D}(x_1,\ldots,x_D)$ is a symmetric polynomial, hence $\Delta_{r,D}\neq 0$ iff $G_{r,D}(\frac{D-1}{D},\ldots,0)\neq 0$.

In the last section, we propose another approach to the initial problem, studied in \cite{lucrare2}, of computing $\pa(n)$ in terms of values of Bernoulli polynomials and Bernoulli Barnes numbers.
In formula $(4.3)$ we show that
$$ \sum_{v=0}^{D-1}(-D)^m d_{\mathbf a,m}(v) = \frac{(-1)^{r-1}D^{m+1}}{m!(r-1-m!)}B_{r-1-m}(\mathbf a) ,\;(\forall)0\leq m\leq r-1.$$
Seeing $d_{\mathbf a,m}(v)$'s as indeterminates and considering also the identities
$$ \sum_{m=0}^{r-1}\sum_{v=1}^{D} d_{\mathbf a,m}(v)D^{n+m} \frac{B_{n+m+1}(\frac{v}{D})}{n+m+1}  = \frac{(-1)^{r-1} n!}{(n+r)!}B_{r+n}(\mathbf a)-\delta_{0n},\;(\forall)0\leq n\leq rD-r-1,$$
we obtain a system of $rD$ linear equations with a determinant $\bar \Delta_{r,D}$. In Remark $4.1$ we note that if $\bar \Delta_{r,D}\neq 0$, then $d_{\mathbf a,m}(v)$,
$0\leq m\leq r-1$, $1\leq v\leq D$, are the unique solutions of the above system.
We consider the polynomial $\bar F_{r,D}\in \mathbb Q[x_1,\ldots,x_D]$ defined by
$$ \bar F_{r,D}(x_1,\ldots,x_D):= \begin{vmatrix}
1 & \cdots & 1 & \cdots & 0 & \cdots & 0 \\
\vdots & \cdots & \vdots & \cdots & \vdots & \cdots & \vdots \\
0 & \cdots & 0 & \cdots & 1 & \cdots & 1 \\
B_1(x_1) & \cdots & B_1(x_D) & \cdots & \frac{B_r(x_1)}{r} & \cdots & \frac{B_r(x_D)}{r} \\
\vdots & \vdots & \vdots & \vdots & \vdots \\
\frac{B_{rD-r}(x_1)}{rD-r} & \cdots & \frac{B_{rD-r}(x_D)}{rD-r} & \cdots &  
\frac{B_{rD-1}(x_1)}{rD-1} & \cdots & \frac{B_{rD-1}(x_D)}{rD-1}
\end{vmatrix}.$$
We have that $\bar{\Delta}_{r,D}=(-D)^{D\binom{r}{2}+\binom{rD-r}{2}} \bar F_{r,D}(\frac{D-1}{D},\ldots,\frac{1}{D},0)$. In formula $(4.9)$ we show that
$$\bar F_{r,D}(x_1,\ldots,x_D)=(-1)^{(D+1)\binom{r}{2}} \left(\prod_{1\leq i<j\leq D}(x_j-x_i)^r\right) G_{r,D}(x_1,\ldots,x_D),$$
where $G_{r,D}$ is a symmetric polynomial with $\deg G_{r,D}\leq r \binom{rD-r}{2} + D\binom{r}{2} - r\binom{D}{2}$.

Using the methods of Olson \cite{olson}, in Proposition $4.2$ we prove that for any $D\geq 1$ we have
$$(1)\; F_{1,D}(x_1,\ldots,x_D)= \frac{1}{(D-1)!}\prod_{1\leq i<j\leq D}(x_j-x_i),\;(2)\; \bar{\Delta}_{1,D} = \frac{1!2!\cdots (D-2)!}{(-D)^D}.$$
By our computer experiments in {\sc Singular} \cite{DGPS}, we expect that the following formula holds
$$\bar F_{r,2}(x_1,x_2)= (-1)^{\binom{r}{2}} \frac{[1!2!\cdots(r-1)!]^3}{r!(r+1)!\cdots(2r-1)!} (x_2-x_1)^r \prod_{j=0}^{r-1}((x_2-x_1)^2-j^2)^{r-j},\;(\forall)r\geq 1,$$
some justifications being noted in Remark $4.3$. Also, we propose a formula for $\bar F_{2,D}$, see Conjecture $4.5$, but we are unable to ``guess'' a formula
for $\bar F_{r,D}$ in general.

\newpage
\section{Preliminaries}

Let $\mathbf a:=(a_1,a_2,\ldots,a_r)$ be a sequence of positive integers, $r\geq 1$. 
The \emph{restricted partition function} associated to $\mathbf a$ is $\pa:\mathbb N \rightarrow \mathbb N$, 
$$\pa(n):= \#\{ (x_1,\ldots,x_r)\in\mathbb N^r\;:\; \sum_{i=1}^r a_ix_i=n\},\;(\forall)n\geq 0.$$
Let $D$ be a common multiple of $a_1,\ldots,a_r$. Bell \cite{bell} has proved that $\pa(n)$ is a 
quasi-polynomial of degree $r-1$, with the period $D$, i.e.
\begin{equation}\label{11}
\pa (n) = d_{\mathbf a,r-1}(n)n^{r-1} + \cdots + d_{\mathbf a,1}(n)n + d_{\mathbf a,0}(n),\;(\forall)n\geq 0,
\end{equation}
where $d_{\mathbf a,m}(n+D)=d_{\mathbf a,m}(n)$, $(\forall)0\leq m\leq r-1,n\geq 0$, and $d_{\mathbf a,r-1}(n)$ is not identically zero.
The \emph{Barnes zeta} function associated to $\mathbf a$ and $w>0$ is
 $$ \za(s,w):=\sum_{n=0}^{\infty} \frac{\pa(n)}{(n+w)^s},\; \Ree s>r,$$
 see \cite{barnes} and \cite{spreafico} for further details. It is well known that $\za(s,w)$ is meromorphic on $\mathbb C$ with poles at most in the set $\{1.\ldots,r\}$.
We consider the function
 \begin{equation}\label{12}
 \za(s) := \lim_{w\searrow 0}(\za(s,w)-w^{-s}).
 \end{equation}
 In \cite[Lemma 1.6]{lucrare} we proved that 
 \begin{equation}\label{13}
 \za(s)=\frac{1}{D^s}\sum_{m=0}^{r-1}\sum_{v=1}^{D} d_{\mathbf a,m}(v)D^m \zeta(s-m,\frac{v}{D}), 
 \end{equation}
 where $\zeta(s,w):=\sum_{n=0}^{\infty} \frac{1}{(n+w)^s},\;\Ree s>1,$
 is the \emph{Hurwitz zeta} function. 
The \emph{Bernoulli numbers} $B_j$ are defined by
$$ \frac{z}{e^z-1} = \sum_{j=0}^{\infty}B_j \frac{z^j}{j!}, $$
$B_0=1$, $B_1=-\frac{1}{2}$, $B_2=\frac{1}{6}$, $B_4=-\frac{1}{30}$ and $B_n=0$ if $n$ is odd and greater than $1$.
The \emph{Bernoulli polynomials} are defined by
$$\frac{ze^{xz}}{(e^z-1)}=\sum_{n=0}^{\infty}B_n(x)\frac{z^n}{n!}.$$
They are related with the Bernoulli numbers by
$$B_n(x)=\sum_{k=0}^n \binom{n}{k}B_{n-k}x^k.$$
The \emph{Bernoulli-Barnes polynomials} are defined by
$$\frac{z^r e^{xz}}{(e^{a_1 z}-1)\cdots (e^{a_r z}-1)}=\sum_{j=0}^{\infty}B_j(x;\mathbf a)\frac{z^j}{j!}. $$
The \emph{Bernoulli-Barnes numbers} are defined by
$$B_j(\mathbf a):=B_j(0;\mathbf a)=\sum_{i_1+\cdots+ i_r=j}\binom{j}{i_1,\ldots,i_r} B_{i_1}\cdots B_{i_r}a_1^{i_1-1}\cdots a_r^{i_r-1}.$$

In \cite[Formula (2.9)]{lucrare2} we proved that
\begin{equation}\label{poola}
\sum_{m=0}^{r-1}\sum_{v=1}^{D} d_{\mathbf a,m}(v)D^{n+m} \frac{B_{n+m+1}(\frac{v}{D})}{n+m+1}  = \frac{(-1)^{r-1} n!}{(n+r)!}B_{r+n}(\mathbf a)-\delta_{0n},\;(\forall)n\in\mathbb N,
\end{equation}
where $\delta_{0n}=\begin{cases} 1,& n=0,\\ 0,& n\geq 1\end{cases}$ is the \emph{Kronecker symbol}.
Given values $0\leq n\leq rD-1$ in $(1.9)$ and seeing $d_{\mathbf a,m}(v)$'s as indeterminates, we obtain a system of linear equations with the determinant \small{
\begin{equation}\label{pista}
\Delta_{r,D}:= \begin{vmatrix}
\frac{B_1(\frac{1}{D})}{1} & \cdots & \frac{B_1(1)}{1} 
& \cdots & D^{r-1}\frac{B_r(\frac{1}{D})}{r} & \cdots & D^{r-1}\frac{B_r(1)}{r} \\
D\frac{B_2(\frac{1}{D})}{2} & \cdots & D\frac{B_2(1)}{2} 
& \cdots & D^{r}\frac{B_{r+1}(\frac{1}{D})}{r+1} & \cdots & D^{r}\frac{B_{r+1}(1)}{r+1} \\
\vdots & \vdots & \vdots & \vdots & \vdots & \vdots & \vdots  \\
D^{rD-1}\frac{B_{rD}(\frac{1}{D})}{rD} & \cdots & D^{rD-1}\frac{B_{rD}(1)}{rD} & 
 \cdots & D^{rD+r-2}\frac{B_{rD+r-1}(\frac{1}{D})}{rD+r-1} & \cdots & D^{rD+r-2}\frac{B_{rD+r-1}(1)}{rD+r-1} 
\end{vmatrix} \end{equation}}
Using basic properties of determinants and the fact that 
$$B_n(1-x)=(-1)^nB_n(x) \text{ for all } n\geq 0,$$ 
it follows that
\begin{equation}\label{pista2}
\Delta_{r,D} = (-1)^{\frac{rD(rD +r)}{2}} D^{\frac{rD(rD +r-2)}{2}} \begin{vmatrix}
\frac{B_1(\frac{D-1}{D})}{1} & \cdots & \frac{B_1(0)}{1}  
& \cdots & \frac{B_r(\frac{D-1}{D})}{r} & \cdots & \frac{B_r(0)}{r} \\
\frac{B_2(\frac{D-1}{D})}{2} & \cdots & \frac{B_2(0)}{2}  
& \cdots & \frac{B_{r+1}(\frac{D-1}{D})}{r+1} & \cdots & \frac{B_{r+1}(0)}{r+1} \\
\vdots & \vdots & \vdots & \vdots & \vdots & \vdots & \vdots  
\\
\frac{B_{rD}(\frac{D-1}{D})}{rD} & \cdots & \frac{B_{rD}(0)}{rD} & 
 \cdots & \frac{B_{rD+r-1}(\frac{D-1}{D})}{rD+r-1} & \cdots & \frac{B_{rD+r-1}(0)}{rD+r-1} 
\end{vmatrix}.
\end{equation}

\begin{prop}(See \cite[Proposition 2.1]{lucrare2} and \cite[Corollary 2.2]{lucrare2})

With the above notations, if $\Delta_{r,D}\neq 0$, then 
$$d_{\mathbf a,m}(v) = \frac{\Delta_{r,D}^{m,v}}{\Delta_{r,D}},\;(\forall) 1\leq v\leq D, 0\leq m\leq r-1,$$ 
where $\Delta_{r,D}^{m,v}$ is the determinant obtained from $\Delta_{r,D}$, as defined in $(\ref{pista})$, by replacing the 
$(mD+v)$-th column with the column $(\frac{(-1)^{r-1} n!}{(n+r)!}B_{n+r}(\mathbf a)-\delta_{n0})_{0\leq n\leq rD-1}$.
Consequently, 
$$\pa(n)=\frac{1}{\Delta_{r,D}}\sum_{m=0}^{r-1} \Delta_{r,D}^{m,v} n^m,\;(\forall)n\in\mathbb N.$$
\end{prop}

\begin{proof}
The first part follows from the Cramer rule applied to the system \eqref{poola}. The second part is a consequence
of the first part and \eqref{11}.
\end{proof}

\begin{obs}
\emph{
In \cite{lucrare2} it was conjectured that $\Delta_{r,D}\neq 0$ for any $r,D\geq 1$. An affirmative answer was given in
the case $r=1$, $r=2$ and $D=1$. In the general case, an equivalent form was 
given in \cite[Theorem 2.3]{lucrare2}, which reduced the problem to show that a $r\times r$ determinant is non zero.
In the next section we tackle this problem from another point of vue, by studying a polynomial $F_{r,D}$ 
is $D$ indeterminates with the property that $\Delta_{r,D}=F_{r,D}(\frac{D-1}{D},\ldots,\frac{1}{D},0)$.}
\end{obs}

\section{Determinants with Bernoulli polynomials}

Let $r,D\geq 1$ be two integers. We consider the polynomial \small{
\begin{equation}\label{31}
 F_{r,D}(x_1,\ldots,x_D):= \begin{vmatrix}
\frac{B_1(x_1)}{1} & \cdots & \frac{B_1(x_D)}{1}  
& \cdots & \frac{B_r(x_1)}{r} & \cdots & \frac{B_r(x_D)}{r} \\
\frac{B_2(x_1)}{2} & \cdots & \frac{B_2(x_D)}{2}  
& \cdots & \frac{B_{r+1}(x_1)}{r+1} & \cdots & \frac{B_{r+1}(x_D)}{r+1} \\
\vdots & \vdots & \vdots & \vdots & \vdots & \vdots & \vdots  
\\
\frac{B_{rD}(x_1)}{rD} & \cdots & \frac{B_{rD}(x_D)}{rD} & 
 \cdots & \frac{B_{rD+r-1}(x_1)}{rD+r-1} & \cdots & \frac{B_{rD+r-1}(x_D)}{rD+r-1} 
\end{vmatrix} \in \mathbb Q[x_1,\ldots,x_D].
\end{equation}}
According to (\ref{pista2}) and \eqref{31}, using the notations from the previous section, we have that
\begin{equation}\label{32}
 \Delta_{r,D}=(-1)^{\frac{rD(rD +r)}{2}}D^{\frac{rD(rD +r-2)}{2}}\cdot F_{r,D}(\frac{D-1}{D},\ldots,\frac{1}{D},0).
\end{equation}

\begin{lema}
For any $r\geq 1$ we have that
$$ \Delta:=\begin{vmatrix}
1 & \frac{1}{2} & \cdots & \frac{1}{r} \\ 
\frac{1}{2} & \frac{1}{3} & \cdots & \frac{1}{r+1} \\ 
\vdots & \vdots & \ddots  & \vdots \\ 
\frac{1}{r} & \frac{1}{r+1} & \cdots & \frac{1}{2r-1}
\end{vmatrix} =  \frac{[1!2!\cdots(r-1)!]^3}{r!(r+1)!\cdots(2r-1)!}.$$
\end{lema}

\begin{proof}
We let $$ \Delta_{\ell}:= \begin{vmatrix}
\frac{r!}{\ell} & \frac{r!}{\ell+1} & \cdots & \frac{r!}{r} \\ 
\frac{(r+1)!}{\ell+1} & \frac{(r+1)!}{\ell+2} & \cdots & \frac{(r+1)!}{r+1} \\ 
\vdots & \vdots & \ddots  & \vdots \\ 
\frac{(2r-\ell)!}{r} & \frac{(2r-\ell)!}{r+1} & \cdots & \frac{(2r-\ell)!}{2r-\ell}
\end{vmatrix}.$$
Note that $\Delta=r!(r+1)!\cdots(2r-1)!\Delta_1$. We have $\Delta_r=(r-1)!$. For $1\leq \ell<r$, we have
$$\Delta_{\ell}= (r-1)!\begin{vmatrix}
\frac{r!}{\ell} & \cdots & \frac{r!}{r-1}  & 1 \\ 
\frac{(r+1)!}{\ell+1} & \cdots & \frac{(r+1)!}{r}  & \frac{r!}{(r-1)!} \\ 
\vdots & \vdots & \ddots  & \vdots \\ 
\frac{(2r-\ell)!}{r} & \cdots & \frac{(2r-\ell)!}{2r-\ell-1}  & \frac{(2r-\ell-1)!}{(r-1)!}
\end{vmatrix}. $$
Multiplying the first line accordingly and adding to the next lines in order to obtain zeroes on the last column, it follows that
$$\Delta_{\ell}= (r-\ell)! \det \left(\frac{k(r+k-1)!(r-s)}{s+k}\right)_{\substack{\ell\leq s\leq r-1 \\ 1\leq k \leq r-\ell}} = 
(r-1)!\frac{((r-\ell)!)^2}{\ell \cdots (r-1)}\Delta_{\ell+1}=((r-\ell)!)^2\ell!\Delta_{\ell+1},$$
hence the induction step is complete.
\end{proof}

\begin{prop}
We have that $$F_{r,1}(x) = \frac{[1!2!\cdots(r-1)!]^3}{r!(r+1)!\cdots(2r-1)!} x^{r^2} + \text{ terms of lower degree.}$$
\end{prop}

\begin{proof}
 We have $B_n(x)=x^n+$ terms of lower order, hence the result follows from Lemma $3.1$.
\end{proof}

\begin{prop}
For $r=1$ and $D\geq 1$ the following hold:
\begin{enumerate}
\item[(1)] There exists a symetric polynomial  $G_{1,D}(x_1,\ldots,x_D)$ of degree $D$ such that
$$F_{1,D}(x_1,\ldots,x_D) = \prod_{1\leq i<j\leq D}(x_j-x_i) G_{1,D}(x_1,\ldots,x_D).$$
\item[(2)] $G_{1,D}(x_1,\ldots,x_D)=\frac{1}{D!}x_1x_2\cdots x_D + $ terms of lower degree.
\item[(3)] $G_{1,D}(0,\ldots,0)=\frac{(-1)^D}{(D+1)!}$.
\end{enumerate}
\end{prop}

\begin{proof}
(1) From \eqref{31} it follows that
\begin{equation}\label{1D}
 F_{1,D}(x_1,\ldots,x_D) = \frac{1}{D!} \begin{vmatrix}
B_1(x_1) & B_1(x_2) & \cdots & B_1(x_D) \\
B_2(x_1) & B_2(x_2) & \cdots & B_2(x_D) \\
\vdots & \vdots & \vdots & \vdots \\
B_D(x_1) & B_D(x_2) & \cdots & B_D(x_D)
\end{vmatrix}.
\end{equation}
Moreover, for any permutation $\sigma\in S_D$, we have that
\begin{equation}\label{coocoo}
F_{1,D}(x_{\sigma(1)},\ldots, x_{\sigma(D)})=\varepsilon(\sigma) F_{1,D}(x_1,\ldots,x_D).
\end{equation}
Since $$(x_j-x_i) | B_{\ell}(x_j)-B_{\ell}(x_i),\;(\forall)1\leq \ell\leq D,\;1\leq i<j\leq D,$$
from $(\ref{1D})$ and $(\ref{coocoo})$ it follows that
\begin{equation}\label{35}
F_{1,D}(x_1,\ldots,x_D) = G_{1,D}(x_1,\ldots,x_D)\cdot \prod_{1\leq i<j\leq D}(x_j-x_i),
\end{equation}
where $G_{1,D}\in \mathbb Q[x_1,\ldots,x_D]$ is a symmetrical polynomial of degree $D$.

(2) The homogeneous component of highest degree of $F_{1,D}$ is
$$ \frac{1}{D!} \begin{vmatrix}
x_1 & x_2 & \cdots & x_D \\
x_1^2 & x_2^2 & \cdots & x_D^2 \\
\vdots & \vdots & \vdots & \vdots \\
x_1^D & x_2^D & \cdots & x_D^D
\end{vmatrix} = \frac{x_1\cdots x_D}{D!} \begin{vmatrix}
1 & 1 & \cdots & 1 \\
x_1 & x_2 & \cdots & x_D \\
\vdots & \vdots & \vdots & \vdots \\
x_1^{D-1} & x_2^{D-1} & \cdots & x_D^{D-1}
\end{vmatrix}= \frac{x_1\cdots x_D}{D!} \prod_{1\leq i<j\leq D}(x_j-x_i),$$
hence $G_{1,D}(x_1,\ldots,x_D)=\frac{1}{D!}x_1\cdots x_D + $ terms of lower order.

(3) For any integers $j\geq 0$ and $1 \leq n \leq D$, we let 
$$L_j(x_1,\ldots,x_n):=\text{ the sum of all monomials of degree }j\text{ in }x_1,\ldots,x_n,$$
i.e. $L_1(x_1,\ldots,x_n)=x_1+\cdots+x_n$, $L_2(x_1,\ldots,x_n)=x_1^2+\cdots+x_n^2+x_1x_2+\cdots+x_{n-1}x_n$, etc. It is easy to check that
\begin{equation}\label{star}
 L_j(x_1,\ldots,x_{n-2},x_n)-L_j(x_1,\ldots,x_{n-1}) = (x_n-x_{n-1})L_{j-1}(x_1,\ldots,x_n).
\end{equation}
We let
$$B_{\ell}(x_1,x_k):=\frac{B_j(x_k)-B_j(x_1)}{x_k-x_1},\;(\forall)1<k\leq D,\ell \geq 1.$$
Inductively, for $1<j\leq k \leq D$ and $\ell\geq 1$, we define
\begin{equation}\label{cacat}
B_{\ell}(x_1,\ldots,x_{j-1},x_k):=\frac{B_{\ell}(x_1,\ldots,x_{j-2},x_k) - B_{\ell}(x_1,\ldots,x_{j-1}) }{x_k-x_{j-1}}.
\end{equation}
We prove by induction on $j\geq 1$ that
\begin{equation}\label{mizerie}
 B_{\ell}(x_1,\ldots,x_{j-1},x_k)= \sum_{t=0}^{\ell-j+1} \binom{\ell}{t+j-1}B_{\ell-j+1-t}L_{t}(x_1,\ldots,x_{j-1},x_k),\;(\forall)1\leq \ell\leq D.
\end{equation}
Indeed, since $B_{\ell}(x)=\sum_{t=0}^{\ell}\binom{\ell}{t}B_{\ell-t}x^t$, it follows that (\ref{mizerie}) holds for $j=1$. Now, assume that $j\geq 2$. 
From the induction hypothesis, (\ref{cacat}), (\ref{star}) and (\ref{mizerie}) it follows that 
$$ B_{\ell}(x_1,\ldots,x_{j-1},x_k) = \sum_{t=1}^{\ell-j+2} \binom{\ell}{t+j-2}B_{\ell-j+2-t}\cdot \frac{L_{t}(x_1,\ldots,x_{j-2},x_k) - L_{t}(x_1,\ldots,x_{j-1})}{x_k-x_{j-1}} = $$
\small{
 $$  \sum_{t=1}^{\ell-j+2} \binom{\ell}{t+j-2}B_{\ell-j+2-t} L_{t-1}(x_1,\ldots,x_{j-1},x_k) = \sum_{t=0}^{\ell-j+1} \binom{\ell}{t+j-1} B_{\ell-j+1-t} L_{t}(x_1,\ldots,x_{j-1},x_k),$$}
hence the induction step is complete. Using standard properties of determinants, from $(3.3)$ it follows that
$$F_{1,D}(x_1,\ldots,x_D) = \frac{1}{D!}\prod_{2\leq j \leq D}(x_j-x_1) \begin{vmatrix}
B_1(x_1) & 1 & \cdots & 1 \\
B_2(x_1) & B_2(x_1,x_2) & \cdots & B_2(x_1,x_D) \\
\vdots & \vdots & \vdots & \vdots \\
B_D(x_1) & B_D(x_1,x_2) & \cdots & B_D(x_1,x_D)\end{vmatrix}  = \cdots $$ 
\begin{equation}\label{minune}
 = \frac{1}{D!}\prod_{1\leq i < j \leq D}(x_j-x_i)\begin{vmatrix}
B_1(x_1) & B_1(x_1,x_2) &  \cdots & B_1(x_1,\ldots,x_D) \\
B_2(x_1) & B_2(x_1,x_2) &  \cdots & B_2(x_1,\ldots,x_D) \\
\vdots & \vdots & \vdots & \vdots  \\
B_D(x_1) & B_D(x_1,x_2) & \cdots & B_D(x_1,\ldots,x_D)\end{vmatrix},
\end{equation}
hence the last determinant is $D!\cdot G_D(x_1,\ldots,x_D)$. Note that (\ref{mizerie}) implies that
\begin{equation}\label{rapt}
B_{\ell}(x_1,\ldots,x_{j})=0,\;(\forall)1\leq \ell\leq j-2\leq D-2,\text{ and }B_{\ell}(x_1,\ldots,x_{\ell+1})=1,\;(\forall)1\leq \ell\leq D-1.
\end{equation}
From (\ref{minune}) and (\ref{rapt}) it follows that $ F_{1,D}(x_1,\ldots,x_D) =$
 \small{
\begin{equation}\label{minune2}
= \frac{1}{D!}\prod_{1\leq i < j \leq D}(x_j-x_i)\begin{vmatrix}
B_1(x_1) & 1 & 0 &  \cdots & 0 \\
B_2(x_1) & B_2(x_1,x_2) &  1 & \cdots & 0 \\
\vdots & \vdots & \vdots & \ddots & \vdots \\
B_{D-1}(x_1) & B_{D-1}(x_1,x_2) & \cdots & B_{D-1}(x_1,\ldots,x_{D-1}) & 1 \\ 
B_D(x_1) & B_D(x_1,x_2) & \cdots & B_{D}(x_1,\ldots,x_{D-1}) & B_D(x_1,\ldots,x_D)\end{vmatrix}.
\end{equation}}
Also, from (\ref{mizerie}), we have $B_{\ell}(0,\ldots,0)=\binom{\ell}{j-1}B_{\ell-j+1}$, hence, from (\ref{35}) and (\ref{minune2}), we get
\begin{equation}
 M_D:=D!G_{1,D}(0,\ldots,0) = \begin{vmatrix}
B_1 & 1 & 0 &  \cdots & 0 \\
B_2 & \binom{2}{1}B_1 &  1 & \cdots & 0 \\
\vdots & \vdots & \vdots & \ddots & \vdots \\
B_{D-1} & \binom{D-1}{1} B_{D-2} & \cdots & \binom{D-1}{D-2}B_1 & 1 \\ 
B_D & \binom{D}{1} B_{D-1} & \cdots & \binom{D}{D-2}B_2 & \binom{D}{D-1}B_1\end{vmatrix}.
\end{equation}
Since $M_D$ is the determinant of a lower Hessenberg matrix, according to \cite[pag.222,Theorem]{cahill}, we have the recursive relation
\begin{equation}\label{onan}
M_D= \binom{D}{D-1}B_1M_{D-1}+\sum_{\ell=1}^{D-1}(-1)^{D-\ell}\binom{D}{D+1-\ell}B_{D+1-\ell}M_{\ell-1},\;(\forall)D\geq 1,\text{ where }M_0:=1.
\end{equation}
We prove that 
\begin{equation}\label{wish}
M_D=\frac{(-1)^D}{D+1}, (\forall)D\geq 1,
\end{equation}
using induction on $D\geq 1$. For $D=1$ we have $M_1=B_1=-\frac{1}{2}$, hence the (\ref{wish}) holds. If $D\geq 2$ then from induction hypothesis
and (\ref{wish}) it follows that 
\begin{equation}\label{porc}
M_D= (-1)^{D-1}B_1 + \sum_{\ell=1}^{D-1}(-1)^{D-\ell}\binom{D}{D-\ell+1}B_{D+1-\ell}\frac{(-1)^{\ell-1}}{\ell}  
= (-1)^{D-1}\sum_{\ell=1}^{D}\frac{1}{\ell}\binom{D}{\ell-1}B_{D+1-\ell}. 
\end{equation}
Since $\binom{D+1}{\ell}=\frac{D+1}{\ell}\binom{D}{\ell-1}, (\forall)1\leq \ell\leq D$, from (\ref{porc}) it follows that
\begin{equation}\label{mangalita}
M_D = \frac{(-1)^{D-1}}{D+1}\sum_{\ell=1}^{D}\binom{D+1}{\ell}B_{D+1-\ell} = \frac{(-1)^{D-1}}{D+1}\left(\sum_{\ell=1}^{D+1}\binom{D+1}{\ell}B_{D+1-\ell} -1\right)
\end{equation}
On the other hand
$$\sum_{\ell=1}^{D+1}\binom{D+1}{\ell}B_{D+1-\ell} = B_{D+1}(1)-B_{D+1}(0) = 0,$$
hence (\ref{mangalita}) completes the induction step. Therefore, we proved $(\ref{wish})$ and thus 
$$G_D(0,\ldots,0)=\frac{(-1)^D}{(D+1)!}, \text{ as required.}$$
\end{proof}
 
For any integer $n\geq 1$, we denote
$$E_{n,0}(x_1,\ldots,x_n):=1, \; E_{n,1}(x_1,\ldots,x_n):=x_1+\cdots+x_n,\;E_{n,n}(x_1,\ldots,x_n):=x_1x_2\cdots x_n,$$
the \emph{elementary symmetric polynomials} in $\mathbb Q[x_1,\ldots,x_n]$.

\begin{teor}
With the above notations, we have that 
$$F_{1,D}(x_1,\ldots,x_D)=\frac{1}{D!}  \prod_{1\leq i<j\leq D}(x_j-x_i) \sum_{t=0}^D (-1)^t \frac{E_{D,D-t}(x_1,\ldots,x_D)}{t+1} .$$
\end{teor}

\begin{proof}
 We use induction on $D\geq 1$. For $D=1$ we have
$$F_{1,1}(x_1)=B_1(x)=x_1-\frac{1}{2}=E_{1,1}(x_1)-\frac{E_{1,0}(x_1)}{2},$$
hence the required formula holds. For $D\geq 2$, from (\ref{1D}) it follows that
\begin{equation}\label{capra}
F_{1,D}(x_1,\ldots,x_D)= \frac{1}{D} \sum_{k=1}^D (-1)^{D+k}B_D(x_k) F_{1,D-1}(x_1,\ldots,\widehat{x_k},\ldots,x_D),  
\end{equation}
where $\widehat{x_k}$ means that the variable $x_k$ is omitted. From the induction hypothesis and (\ref{capra}) it follows that
\small{
\begin{equation}\label{capraa}
F_{1,D}(x_1,\ldots,x_D)= \frac{1}{D!} \sum_{k=1}^D (-1)^{D+k}B_D(x_k) \prod_{\substack{1\leq i<j\leq D \\ i,j\neq k}}(x_j-x_i)
\sum_{\ell=0}^{D-1} (-1)^{\ell} \frac{E_{D-1,D-1-\ell}(x_1,\ldots,\widehat{x_k},\ldots,x_D)}{\ell+1}.
\end{equation}}
The relation (\ref{capraa}) is equivalent to
\begin{equation}\label{caprita}
 \frac{D!F_{1,D}(x_1,\ldots,x_D)}{\prod_{1\leq i<j\leq D}(x_j-x_i)} = \sum_{k=1}^D (-1)^{D-1}\frac{1}{\prod_{j\neq k}(x_j-x_k)} B_D(x_k)
\sum_{\ell=0}^{D-1} (-1)^{\ell} \frac{E_{D-1,D-1-\ell}(x_1,\ldots,\widehat{x_k},\ldots,x_D)}{\ell+1}.
\end{equation}
From (\ref{caprita}), in order to complete the proof it is enough to show that
\begin{equation}\label{iada}
\sum_{t=0}^D (-1)^t \frac{E_{D,D-t}(x_1,\ldots,x_D)}{t+1} = 
\sum_{k=1}^D \sum_{\ell=0}^{D-1} 
\frac{(-1)^{D-1-\ell}B_D(x_k)E_{D-1,D-1-\ell}(x_1,\ldots,\widehat{x_k},\ldots,x_D)}{(\ell+1)\prod_{j\neq k}(x_j-x_k)}.
\end{equation}
Since $B_D(x_k)=\sum_{s=0}^D \binom{D}{s}B_{D-s}x^s$, it follows that (\ref{iada}) is equivalent to
\begin{equation}\label{iedu}
(-1)^t \frac{E_{D,D-t}(x_1,\ldots,x_D)}{t+1} = \sum_{k=1}^D \sum_{\ell=0}^{\min\{t,D-1\}} 
\frac{(-1)^{D-1-\ell}\binom{D}{t-\ell}B_{t-\ell}x_k^{D-t+\ell} 
E_{D-1,D-1-\ell}(x_1,\ldots,\widehat{x_k},\ldots,x_D) }{(\ell+1)\prod_{j\neq k}(x_j-x_k)},
\end{equation}
for any $0\leq t\leq D$. Since, by Proposition $3.3(3)$, we have that
$$ \frac{D!F_{1,D}(x_1,\ldots,x_D)}{\prod_{1\leq i<j\leq D}(x_j-x_i)}\rvert_{x_1=\cdots=x_D=0} = \frac{(-1)^D}{D+1} = \frac{(-1)^D E_{D,0}(x_1,\ldots,x_D)}{D+1},$$
it is enough to prove (\ref{iedu}) for $0\leq t\leq D-1$. Similarly, by Proposition $3.2(2)$ we can dismiss the case $t=0$. Assume in
the following that $1\leq t\leq D-1$. As the both sides in $(\ref{iedu})$ are symmetric polynomials, it is enough to
prove that $(\ref{iedu})$ holds when we evaluate it in $x_{D-t+1}=\cdots=x_D=0$. Moreover, in this case, 
$E_{D-1,D-1-\ell}(x_1,\ldots,\widehat{x_k},\ldots,x_D)=0$ for any $\ell<t$. Therefore, $(\ref{iedu})$ is equivalent to
$$(-1)^t\frac{x_1\cdots x_{D-t}}{t+1} =  \sum_{k=1}^{D-t} 
\frac{(-1)^{t}x_k^D x_1\cdots \widehat{x_k} \cdots x_{D-t}}{(t+1)\prod_{j\neq k,\;j\leq D-t}(x_k-x_j)x_k^t}, $$
hence it is equivalent to 
$$\sum_{k=1}^{D-t} \frac{x_k^{D-t-1}}{\prod_{j\neq k,\;j\leq D-t}(x_k-x_j)}=1,$$
which can be easily proved by expanding a Vandermonde determinant of order $D-t$.
\end{proof}

\begin{cor}
We have that $$\Delta_{1,D}=(-1)^{\frac{D(D+1)}{2}}\frac{(D-1)!(D-2)!\cdots 1!}{D!}
\sum_{t=0}^D (-1)^t\frac{E_{D,D-t}(\frac{D-1}{D},\ldots,\frac{1}{D},0)}{t+1}.$$
\end{cor}

\begin{proof}
From (3.2) and Theorem $3.4$ it follows that
\begin{equation}\label{c51}
\Delta_{1,D}=(-1)^{\frac{D(D+1)}{2}}D^{\frac{D(D-1)}{2}}\frac{1}{D!}\prod_{1\leq i<j\leq D}\left(\frac{j-i}{D}\right)
\sum_{t=0}^D (-1)^t\frac{E_{D,D-t}(\frac{D-1}{D},\ldots,\frac{1}{D},0)}{t+1}.
\end{equation}
On the other hand
\begin{equation}\label{c52}
\prod_{1\leq i<j\leq D}\left(\frac{j-i}{D}\right) = \frac{(D-1)!(D-2)!\cdots 1!}{D^{\frac{D(D-1)}{2}}},
\end{equation}
hence, from (\ref{c51}) and (\ref{c52}) we get the required result.
\end{proof}

Unfortunately, in the general, it seems to be very difficult to give an exact formula for $F_{r,D}(x_1,\ldots,x_D)$.
What it is easy to show is the following generalization of Proposition $3.3(1)$. 

\begin{prop}
For any integers $r,D\geq 1$, there exists a symmetric polynomial $G_{r,D}$ of degree $\leq r^2\binom{D+1}{2}-r\binom{D}{2}$ such that
$$F_{r,D}(x_1,\ldots,x_D) = \prod_{1\leq i<j\leq D}(x_j-x_i)^r G_{r,D}(x_1,\ldots,x_D),$$
where, with the notations from (3.7), we have that
$$G_{r,D}(x_1,\ldots,x_D) = \begin{vmatrix}
\frac{B_1(x_1)}{1} & \cdots & \frac{B_1(x_1,\ldots,x_D)}{1}  
& \cdots & \frac{B_r(x_1)}{r} & \cdots & \frac{B_r(x_1,\ldots,x_D)}{r} \\
\frac{B_2(x_1)}{2} & \cdots & \frac{B_2(x_1,\ldots,x_D)}{2}  
& \cdots & \frac{B_{r+1}(x_1)}{r+1} & \cdots & \frac{B_{r+1}(x_1,\ldots,x_D)}{r+1} \\
\vdots & \vdots & \vdots & \vdots & \vdots & \vdots & \vdots  
\\
\frac{B_{rD}(x_1)}{rD} & \cdots & \frac{B_{rD}(x_1,\ldots,x_D)}{rD} & 
 \cdots & \frac{B_{rD+r-1}(x_1)}{rD+r-1} & \cdots & \frac{B_{rD+r-1}(x_1,\ldots,x_D)}{rD+r-1} 
\end{vmatrix}.$$
\end{prop}

\begin{proof}
Using standard properties of determinants, as in the proof of formula $(3.9)$, we get the required decomposition.
The fact that $G_{r,D}(x_1,\ldots,x_D)$ is symmetric follows from the identity
$$F_{r,D}(x_{\sigma(1)},\ldots,x_{\sigma(D)})=\varepsilon(\sigma)^r F_{r,D}(x_1,\ldots,x_r),\;(\forall)\sigma\in S_D$$
and the decomposition $F_{r,D}(x_1,\ldots,x_D) = \prod_{1\leq i<j\leq D}(x_j-x_i)^r G_{r,D}(x_1,\ldots,x_D)$.
\end{proof}

\section{An approach to compute $\pa(n)$}

Let $\mathbf a:=(a_1,a_2,\ldots,a_r)$ be a sequence of positive integers, $r\geq 1$. 
Let $D$ be a common multiple of $a_1,\ldots,a_r$. Using the notations and definitions from the second section, 
according to \cite[Proposition $2.4$]{lucrare} and \eqref{13}, the function $\za(s)$ is meromorphic in the
 whole complex plane with poles at most in the set $\{1,\ldots,r\}$ which are all simple with residues
\begin{equation}
R_{m+1} =\Res_{s=m+1} \za(s)=\frac{1}{D}\sum_{v=0}^{D-1}d_{\mathbf a,m}(v),\;(\forall)0\leq m\leq r-1.
\end{equation}
On the other hand, according to \cite[Theorem $2.10$]{lucrare} or \cite[Formula (3.9)]{rui} and \eqref{12}, we have that
\begin{equation}
R_{m+1} = \frac{(-1)^{r-1-m}}{m!(r-1-m!)}B_{r-1-m}(a_1,\ldots,a_r) ,\;(\forall)0\leq m\leq r-1.
\end{equation}
It follows that
\begin{equation}
\sum_{v=0}^{D-1}(-D)^m d_{\mathbf a,m}(v) = \frac{(-1)^{r-1}D^{m+1}}{m!(r-1-m!)}B_{r-1-m}(\mathbf a) ,\;(\forall)0\leq m\leq r-1.
\end{equation}
On the other hand, from $(1.9)$ it follows that
\begin{equation}
\sum_{m=0}^{r-1}\sum_{v=1}^{D} d_{\mathbf a,m}(v)D^{n+m} \frac{B_{n+m+1}(\frac{v}{D})}{n+m+1}  = \frac{(-1)^{r-1} n!}{(n+r)!}B_{r+n}(\mathbf a),\;(\forall)0\leq n\leq rD-r-1.
\end{equation}
If we see $d_{\mathbf a,m}(v)$ as indeterminates, (4.3) and (4.4) form a system of linear equations with the determinant 
\small{
\begin{equation}
\bar{\Delta}_{r,D}:=
 \begin{vmatrix}
1 & \cdots & 1 & \cdots & 0 & \cdots & 0 \\
\vdots & \cdots & \vdots & \cdots & \vdots & \cdots & \vdots \\
0 & \cdots & 0 & \cdots & (-D)^{r-1} & \cdots & (-D)^{r-1} \\
B_1(\frac{1}{D}) & \cdots & B_1(1) & \cdots & \frac{D^{r-1}B_r(\frac{1}{D})}{r} & \cdots & \frac{D^{r-1}B_r(1)}{r} \\
\vdots & \vdots & \vdots & \vdots & \vdots \\
\frac{D^{rD-r-1}B_{rD-r}(\frac{1}{D})}{rD-r} &  \cdots & \frac{D^{rD-r-1}B_{rD-r}(1)}{rD-r} & \cdots &  
\frac{D^{rD-2}B_{rD-1}(\frac{1}{D})}{rD-1} & \cdots & \frac{D^{rD-2}B_{rD-1}(1)}{rD-1}
\end{vmatrix}.
\end{equation}}
From (4.5) and the identity $B_n(1-x)=(-1)^nB_n(x)$ it follows that 
\begin{equation}
\bar{\Delta}_{r,D}=(-D)^{D\binom{r}{2}+\binom{rD-r}{2}} \cdot  \begin{vmatrix}
1 & \cdots & 1 & \cdots & 0 & \cdots & 0 \\
\vdots & \cdots & \vdots & \cdots & \vdots & \cdots & \vdots \\
0 & \cdots & 0 & \cdots & 1 & \cdots & 1 \\
B_1(\frac{D-1}{D}) & \cdots & B_1(0) & \cdots & \frac{B_r(\frac{D-1}{D})}{r} & \cdots & \frac{B_r(0)}{r} \\
\vdots & \vdots & \vdots & \vdots & \vdots \\
\frac{B_{rD-r}(\frac{D-1}{D})}{rD-r} &  \cdots & \frac{B_{rD-r}(0)}{rD-r} & \cdots &  
\frac{B_{rD-1}(\frac{D-1}{D})}{rD-1} & \cdots & \frac{B_{rD-1}(0)}{rD-1}
\end{vmatrix}.
\end{equation}

\begin{obs}
\emph{Similarly to Proposition $2.1$, if $\bar{\Delta}_{r,D}\neq 0$, then $d_{\mathbf a,m}(v)$, $0\leq m\leq r-1$, $1\leq v\leq D$ are
the solutions of the system of linear equations consisting in $(4.3)$ and $(4.4)$.}
\end{obs}

Now we consider the polynomial $\bar F_{r,D}\in \mathbb Q[x_1,\ldots,x_D]$ defined as
\begin{equation}
\bar F_{r,D}(x_1,\ldots,x_D):= \begin{vmatrix}
1 & \cdots & 1 & \cdots & 0 & \cdots & 0 \\
\vdots & \cdots & \vdots & \cdots & \vdots & \cdots & \vdots \\
0 & \cdots & 0 & \cdots & 1 & \cdots & 1 \\
B_1(x_1) & \cdots & B_1(x_D) & \cdots & \frac{B_r(x_1)}{r} & \cdots & \frac{B_r(x_D)}{r} \\
\vdots & \vdots & \vdots & \vdots & \vdots \\
\frac{B_{rD-r}(x_1)}{rD-r} & \cdots & \frac{B_{rD-r}(x_D)}{rD-r} & \cdots &  
\frac{B_{rD-1}(x_1)}{rD-1} & \cdots & \frac{B_{rD-1}(x_D)}{rD-1}
\end{vmatrix}
\end{equation}
From (4.6) and (4.7) it follows that
\begin{equation}
\bar{\Delta}_{r,D}=(-D)^{D\binom{r}{2}+\binom{rD-r}{2}} \bar F_{r,D}(\frac{D-1}{D},\ldots,\frac{1}{D},0).
\end{equation}
Note that if $D=1$ then (4.5) and (4.7) implies
$$\bar{\Delta}_{r,1}=(-D)^{\binom{r}{2}} \text{ and } \bar{F}_{r,1}(x_1,\ldots,x_r)=1,$$
therefore, in the following we assume $D\geq 2$.

Using elementary operations in (4.7) and the notations (3.7) it follows that $\bar F_{r,D}(x_1,\ldots,x_D)=$
\small{
$$ = (-1)^{(D+1)\binom{r}{2}} \prod_{2\leq j\leq D}(x_D-x_1)^r \begin{vmatrix}
B_1(x_1,x_2) \;  \cdots & B_1(x_1,x_D) \; \cdots & \frac{B_r(x_1,x_2)}{r} \; \cdots & \frac{B_r(x_1,x_D)}{r} \\
\vdots & \vdots & \vdots & \vdots \\
\frac{B_{rD-r}(x_1,x_2)}{rD-r} \; \cdots & \frac{B_{rD-r}(x_1,x_D)}{rD-r} \; \cdots &  
\frac{B_{rD-1}(x_1,x_2)}{rD-1} \; \cdots & \frac{B_{rD-1}(x_1,x_D)}{rD-1}
\end{vmatrix} = $$
\begin{equation}\label{kookoo}
(-1)^{(D+1)\binom{r}{2}} \prod_{1\leq i<j\leq D}(x_j-x_i)^r \begin{vmatrix}
B_1(x_1,x_2) \; \cdots & B_1(x_1,\ldots,x_D) \; \cdots & \frac{B_r(x_1,x_2)}{r} \; \cdots & \frac{B_r(x_1,\ldots,x_D)}{r} \\
\vdots & \vdots & \vdots & \vdots  \\
\frac{B_{rD-r}(x_1,x_2)}{rD-r} \; \cdots & \frac{B_{rD-r}(x_1,\ldots,x_D)}{rD-r} \; \cdots &  
\frac{B_{rD-1}(x_1,x_2)}{rD-1} \; \cdots & \frac{B_{rD-1}(x_1,\ldots,x_D)}{rD-1}
\end{vmatrix}
\end{equation}}
We denote the last determinant in (\ref{kookoo}) with $\bar G_{r,D}(x_1,\ldots,x_D)$ and we note that $\bar G_{r,D}$ is a symmetric
polynomial with $$\deg(\bar G_{r,D})\leq r \binom{rD-r}{2} + D\binom{r}{2} - r\binom{D}{2}.$$

\begin{prop}
For any $D\geq 2$ we have that $$(1)\; F_{1,D}(x_1,\ldots,x_D)= \frac{1}{(D-1)!}\prod_{1\leq i<j\leq D}(x_j-x_i),\;(2)\; \bar{\Delta}_{1,D} = \frac{1!2!\cdots (D-2)!}{(-D)^D}.$$
\end{prop}

\begin{proof}
(1) Using the method from \cite[Page 262]{olson}, we get
$$\bar F_{1,D}(x_1,\ldots,x_D)=\frac{1}{(D-1)!}\begin{vmatrix}
1 &  \cdots & 1 \\
B_1(x_1) & \cdots & B_1(x_D) \\
\vdots & \vdots & \vdots \\
B_{D-1}(x_1) & \cdots & B_{D-1}(x_D)
\end{vmatrix} = $$
\begin{equation}
= \frac{1}{D!} 
\begin{vmatrix} 
B_0 & 0 & \cdots & 0 \\
\binom{2}{1}B_1 & B_0 & \cdots & 0 \\
\vdots & \vdots & \vdots & \vdots \\
\binom{D-1}{1}B_{D-1} & \binom{D-1}{2}B_2 & \cdots & B_0
\end{vmatrix}
\begin{vmatrix} 
1 &  \cdots & 1 \\
x_1 & \cdots & x_D \\
\vdots & \vdots & \vdots \\
x_1^{D-1} & \cdots & x_D^{D-1}
\end{vmatrix} = \frac{1}{(D-1)!}\prod_{1\leq i<j\leq D}(x_j-x_i).
\end{equation}
(2) The last identity follows $(1)$, $(4.8)$ and $(4.9)$.
\end{proof}

\begin{obs}
For $D=2$, according to (4.9) we have that
\begin{equation}\label{curvar}
\bar F_{r,2}(x_1,x_2)=(-1)^{\binom{r}{2}} (x_2-x_1)^r \bar G_{r,2}(x_1,x_2),\text{ where }
\end{equation}
\begin{equation}\label{ob1}
\bar G_{r,2}(x_1,x_2)=
\begin{vmatrix}
B_1(x_1,x_2) & \frac{B_2(x_1,x_2)}{2} &  \cdots & \frac{B_r(x_1,x_2)}{r} \\
\vdots & \vdots & \vdots & \vdots  \\
\frac{B_r(x_1,x_2)}{r} & \frac{B_{r+1}(x_1,x_2)}{r+1} &  \cdots &  \frac{B_{2r-1}(x_1,x_2)}{2r-1}
\end{vmatrix}.
\end{equation}
On the other hand, according to (3.8), we have that
\begin{equation}\label{ob2}
B_k(x_1,x_2) = \sum_{t=0}^{k-1} \binom{k}{t+1}B_{k-1-t}\sum_{s=0}^t x_1^{t-s}x_2^t,\;(\forall)1\leq k\leq 2r-1.
\end{equation}
In particular, from (\ref{ob1}) and (\ref{ob2}) it follows that
\begin{equation}\label{ob3}
\bar G_{r,2}(x_1,0) = \begin{vmatrix}
1 & \frac{x_1}{2} &  \cdots & \frac{x^{r-1}}{r} \\
\vdots & \vdots & \vdots & \vdots  \\
\frac{x_1^{r-1}}{r} & \frac{x_1^r}{r+1} &  \cdots &  \frac{x_1^{2r-2}}{2r-1}
\end{vmatrix} + \text{ terms of lower degree}. 
\end{equation}
From Lemma $3.1$ and (\ref{ob3}) it follows that
$$\bar G_{r,2}(x_1,0)=\frac{[1!2!\cdots(r-1)!]^3}{r!(r+1)!\cdots(2r-1)!}x_1^{r(r-1)}+ \text{ terms of lower degree}.$$
\end{obs}

Our computer experiments in {\sc Singular} \cite{DGPS} and Remark $4.3$ yield us to the following:

\begin{conj}
For any $r\geq 1$, it holds that
$$\bar F_{r,2}(x_1,x_2)= (-1)^{\binom{r}{2}} \frac{[1!2!\cdots(r-1)!]^3}{r!(r+1)!\cdots(2r-1)!} (x_2-x_1)^r \prod_{j=0}^{r-1}((x_2-x_1)^2-j^2)^{r-j}.$$
\end{conj}

We checked Conjecture $4.4$ for $r\leq 4$ and we are convinced that the formula holds in general. 
Our computer experiments in {\sc Singular} \cite{DGPS} yield us also to the following:

\begin{conj}
For any $D\geq 2$, it holds that 
$$\bar F_{2,D}(x_1,\ldots,x_D)=K(D) \prod_{1\leq i<j\leq D}(x_j-x_i)^2 \sum_{1\leq i < j\leq D} ((x_j-x_i)^2-1),$$
where $K(D)\in \mathbb Q$. Moreover, $K(D)\neq 0$, hence $\bar \Delta_{2,D}\neq 0$.
\end{conj}

We checked Conjecture $4.5$ for $D\leq 4$ and we believe it is true in general. 
Unfortunately, we are not able to ``guess'' a general formula for $\bar F_{r,D}$, the situation being wild even for $D=r=3$ as
$\bar G_{3,3}$ is an irreducible polynomial of degree $18$.


{}

\vspace{2mm} \noindent {\footnotesize
\begin{minipage}[b]{15cm}
Mircea Cimpoea\c s, Simion Stoilow Institute of Mathematics, Research unit 5, P.O.Box 1-764,\\
Bucharest 014700, Romania, E-mail: mircea.cimpoeas@imar.ro
\end{minipage}}
\end{document}